\documentclass[11pt]{article}
\usepackage[utf8]{inputenc}
\usepackage[english]{babel}
\usepackage{amsmath}
\usepackage[document]{ragged2e}
\usepackage{emptypage}
\usepackage{algorithm}
\usepackage{algorithmic}
\usepackage{anyfontsize}
\usepackage{mathtools}
\usepackage{amssymb}
\usepackage{subfig}
\usepackage{graphics}
\usepackage{enumitem}
\usepackage{amsthm}
\DeclareMathOperator*{\argmin}{\arg\!\min}
\theoremstyle{plain} 

\newtheorem{proposition}{Proposition}[section]
\theoremstyle{remark}
\newtheorem{remark}{Remark}[section]
\newtheorem{example}{Example}[section]

\theoremstyle{definition}
\newtheorem{definition}{Definition}
\theoremstyle{plain} 
\newtheorem{lemma}{Lemma}[section]
\theoremstyle{plain}
\newtheorem{corollary}{Corollary}[section]
\usepackage[left=1.5in, right=1in, top=1in, bottom=1in, includefoot, headheight=13.6pt, twoside]{geometry}

\title{A novel approach to find the minimum of real functions and an anomalous test function}
\author{Eddie Conti}
\date{July 2023}

\begin{document}
\maketitle

\begin{abstract} 
\noindent
    The aim of this paper is to present an original approach that takes advantage from the geometric features of strictly convex functions to tackle the problem of finding the minimum from another perspective. The general idea is that near the point of minimum, the function is "v-shaped" and so we can reduce the interval where the minimum lies by finding the intersection between the function and a proper horizontal line whose levels decrease step by step. This idea, under some appropriate assumptions, led us to formalise an algorithm that is able to determine the minimum point sought. Furthermore, we see that this approach can be generalized to a wider class of functions. In the last part of this paper we provide the construction of an anomalous function for which the algorithm cannot be used. 
\end{abstract}

\section{A geometric approach to find the minimum point}
Let us consider a function $f \colon \Omega \to \mathbb{R}$, where $\Omega\subset \mathbb{R}$ and $\Omega$ is connected. Let us state the first result which will be useful in the following pages.
\begin{lemma} \label{Ilmiolemma}
    Let us consider $f\colon \Omega \to \mathbb{R}$ a non constant convex function. There cannot exist two intervals of the form $(x_1,x_2)$, $(x_3,x_4)$ with $x_1< x_2\leq x_3< x_4$ such that the function has a change of monotonicity in each interval.
\end{lemma}
\begin{proof}
    Recall that for a convex function, the function
    \[R(x,y)=\frac{f(y)-f(x)}{y-x}\]
    is monotonically non-decreasing in $y$, for every $x$ fixed, or vice versa. In light of the above, fix $x_1$ and consider with no loss of generality $3$ points $r<s<t$ in $(x_3,x_4)$ such that $f(r)>f(s)$ and $f(s)<f(t)$ (the case where the function first increases and then decreases is analogous). Let be $s$ the point where we have a change of monotonicity. In our scenario, we obtain
    \[R(x_1,r)=\frac{f(r)-f(x_1)}{r-x_1}>\frac{f(s)-f(x_1)}{s-x_1}=R(x_1,s)\]
    and 
    \[R(x_1,s)=\frac{f(s)-f(x_1)}{s-x_1}<\frac{f(t)-f(x_1)}{s-x_1}=R(x_1,t)\cdot \frac{t-x_1}{s-x_1}.\]
    Since $t$ can be arbitrarily close to $s$ and $f(s)<f(t)$, and since the equality cannot occur
    \[
    R(x_1,s)< R(x_1,t).
    \]
    which contradicts the monotonicity of $R(x,y)$.
\end{proof} \noindent
In other terms if the function $f$ is convex then has at most one monotonicity change. This, of course, applies also to strictly convex functions. The next result captures the geometric features of strictly convex functions.

\begin{proposition} \label{proposizmia}
    Let $f \colon \Omega \to \mathbb{R}$ a continuous convex coercive function. Then $f$ admits a global minimum $y^*$. If furthermore $f$ is strictly convex, then every horizontal line of the form $h(x)=y$, where $y\in (y^*,+\infty)$, intersects the graph of $f(x)$ in exactly $2$ points. 
\end{proposition}
\begin{proof}
First of all let us prove that the function $f$ admits a global minimum in $\Omega$. Let us consider, for every $\lambda \in \mathbb{R}$, the $\lambda-$sublevel 
\[L_{\lambda}(f)=\{x\in \Omega\,|\, f(x)\leq \lambda\},\]
our aim is to prove that this set is compact. $L_{\lambda}(f)$ is clearly closed as a consequence of the the Theorem of sign permanence: if we consider $\bar{x} \notin L_{\lambda}(f)$, then $f(\bar{x})>\lambda$ and so the function $g(x)=f(x)-\lambda$ satisfies $g(\bar{x})>0$. From the continuity of $g(x)$ there exists a neighbourhood $U$ of $\bar{x}$ such that $g(x)>0$ for any $x \in U$. This is equivalent to $f(x)>\lambda$ for any $x \in U$. To complete the first part of the proposition we have to prove that $L_{\lambda}(f)$ is limited. If so, the $\lambda-$sublevel will be compact and so the function $f$ admits a global minimum in $L_{\lambda}(f)$ and so in $\Omega$. \\
Suppose by contradiction that $L_{\lambda}(f)$ is not limited, therefore there exists a sequence $x_n \in L_{\lambda}(f)$ such that $|x_n| \to \infty$ for $n \to \infty$. Hence, 
\[ \lim_{n\to \infty} f(x_n)\leq \lambda \]
which contradicts the hypothesis of coercivity of $f$. \\ 
To conclude the proof, let us consider $h(x)=y$ with $y\in (y^*,+\infty)$ and  prove that $h(x)$ intersects $f(x)$ in exactly $2$ points. By continuity and since 
\[\lim_{x  \to \infty} f(x)=+\infty\]
the function reaches all values 
in $(y^*,+\infty)$. Therefore there is at least one intersection. However, by the coercivity, also holds 
\[\lim_{x \to -\infty} f(x)=+\infty\]
and so there must be another intersection, but since the function is strictly convex in $\Omega$, another intersection is not allowed. Indeed, if we denote with $x_1,x_2,x_3$ the three intersections, then the function presents a change of monotonicity in both $(x_1,x_2)$ and $(x_2,x_3)$, which is in contrast with Lemma \ref{Ilmiolemma}. 
\end{proof}

\begin{remark}
The second part of the proposition requires that $f$ is strictly convex. This hypothesis is important because otherwise we can easily find functions for which our approach does not hold. For instance, if we consider $f(x)=k$ a constant function, we are not able to find the two intersections with an horizontal line. \\
In addition, the previous result can be generalized if we work locally, with the effect of lightening the assumptions. We will formalize this in the next pages.
\end{remark} 
\noindent
Proposition \ref{proposizmia} holds also for strongly convex functions in view of the following lemma and recalling that a strongly convex function is also strictly convex:
\begin{lemma}
    If $f$ is a strongly convex function on $\mathbb{R}^n$ then $f$ is coercive. Furthermore, if $g\colon \mathbb{R}^n \to \mathbb{R}$ is a differentiable convex function then for every $\epsilon>0$, $g(x)+\epsilon \lVert x\rVert^2$ is coercive. 
\end{lemma}
\begin{proof}
We recall that $f$ is a strongly convex function if for every $x,y \in \mathbb{R}^n$, 
\[ f(y)\geq f(x)+\nabla f(x)^T(y-x)+\frac{m}{2}\lVert x-y\rVert^2 \]
for a certain positive real number $m$. Let us fix $x$ and consider $y_n$ such that $\lVert y_n\rVert \to \infty$ as $n \to \infty$. Therefore,
\[ f(y_n) \geq f(x)+\nabla f(x)^T(y_n-x)+\frac{m}{2}\lVert x-y_n\rVert^2.\]
We observe that $f(x)$ and $\nabla f(x)$ are fixed numbers and $m/2>0$, as a consequence, the term $\lVert x-y_n\rVert^2$ determines the behavior of the right-hand side. Since $\lVert x-y_n\rVert^2\to \infty$ we conclude that $f(y_n)\to \infty$. Hence, the function $f$ is coercive. \\
To prove the second part, we remind that since $g$ is differentiable, for any $x,y \in \mathbb{R}^n$
\[ g(y)\geq g(x)+\nabla g(x)^T(y-x). \]
If we choose $x=0$ we get that $g(y)\geq g(0)+\nabla g(0)^T(y)$, hence
\[g(y)+\epsilon  \lVert y\rVert^2\geq g(0)+\nabla g(0)^T(y)+\epsilon  \lVert y\rVert^2\]
and the right-hand side diverges to $\infty$ as $\lVert y\rVert\to \infty$.
\end{proof}
\noindent
The second part of the lemma, essentially states how the notion of coercivity and strong convexity are related. \\
Let us introduce the following notion, which will be very useful afterwards. 
\begin{definition}
    Let us consider a continuous function $f \colon \Omega \to \mathbb{R}$. We say that $f$ is locally coercive at $x^*$ if there exists a neighborhood $(a,b)$ of $x^*$ such that 
    \[ \lim_{x\to a^+} f(x)=\lim_{x\to b^-} f(x)=L \in \mathbb{R}.\]
\end{definition}
\noindent
The previous definition is motivated by the following result:
\begin{proposition} \label{proposizionemiacorc}
    Let us consider a strictly convex function $f \colon \Omega \to \mathbb{R}$ which admits a minimum at $x^*$, then $f$ is continuous and $f$ is locally coercive at $x^*$.
\end{proposition}
\begin{proof}
    First of all, let us prove the continuity of $f$. Recalling that a convex function $f$ is locally Lipschitz, then we conclude the continuity of $f$. \\
    For the second part, since $f$ is continuous and $x^*$ is a global minimum (since strictly convex functions have unique global solutions), there exists a neighborhood $B(x^*,\delta)$ such that $f(x)>f(x^*)$ for every $x\in B(x^*,\delta)$. Let us consider $b\in U(x^*,\delta)$ such that $b>x^*$ and $f(b)>f(x^*)$, by the continuity of the function, $f$ takes every value between $[f(x^*),f(b)]$. Furthermore, there exists $a<x^*$ such that $f(a)>f(x^*)$ and the function $f$ attains each value in $[f(x^*),f(a)]$. Therefore, unless we rename $a$ and $b$ it is possible to choose a neighborhood $(a,b)$ such that $f$ is locally coercive.
\end{proof}

\begin{remark}
    The previous proposition holds also if $f$ is strictly convex in a neighborhood  $B(x^*,\delta)$ of the minimum point. As a consequence, proposition \ref{proposizmia} can be adapted to work in this case. We have that by continuity and local coercivity, there exists an $\bar{y}$ such that $h(x)=y$ where $y\in (y^*,\bar{y})$ intersects the graph of $f(x)$ in $2$ points. In other terms we have rewritten in local terms the content of Proposition \ref{proposizmia}
\end{remark}
\noindent
We are almost ready to present the idea of the algorithm for the minimum. Let us introduce the following definition.
\begin{definition} \label{definizionesimm}
    Let us assume that $f \colon \Omega \to \mathbb{R}$ is a continuous strictly convex coercive function. Let us denote $x_{y,i}$ for $i=1,2$ the two points in $\Omega$ such that $f(x_{y,i})=y$ for any given $y\in (y^*,+\infty)$. We will call the \textit{symmetry curve of f} the curve $S_f$ whose graph is defined as
\begin{equation} \label{symmetrycurve}
    G(S_f):=\Bigl\{\Bigl(\frac{x_{y,1}+x_{y,2}}{2},y\Bigr)\,|\, y\in (y^*,+\infty)\Bigr\}.  
\end{equation}
\end{definition} \noindent
\begin{lemma} \label{lemma_Sf}
The symmetry curve of $f$ is defined in $(y^*,+\infty)$ and can be extended to $[y^*,+\infty)$ by defining
\[ x^*=\argmin_{x\in \Omega} f(x),\]
since $(x^*,y^*)\in G(S_f)$.
In other words, $S_f$ intersects $f$ at exactly one point, which is the minimum point. 
\end{lemma}
\begin{proof}
    Let $x^*$ be the minimum point of $f$, then $h(x)=y^*$ is tangent to $f(x)$ at $(x^*,y^*)$. If we add that point to $G(S_f)$ we obtain that $S_f$ intersects $f$ at exactly one point. This is possible since if we let $x^1_y\to x^*$ and $x^2_y \to x^*$ then $y\to y^*$ because $f$ is continuous and strictly convex. As a consequence
    \[ \Bigl(\frac{x^1_y+x^2_y}{2},y\Bigr) \to (x^*,y^*),\]
    hence we can extend $S_f$ including the point where $f(x)$ attains its minimum.
\end{proof} 
\begin{remark}
    As already pointed out, Definition \ref{definizionesimm} and Proposition \ref{proposizmia} can be generalized to a function that is strictly convex in a neighborhood of the minimum point.
\end{remark}
\noindent
It is important to underline that $G(S_f)$ lies in the epigraph of $f(x)$ and cannot be extended outside it. This is because $S_f$ is defined by taking the middle point of $x_{y,1},x_{y,2}$ with $y\in [y^*,+\infty)$. \\
\section{Formalization of the algorithm}

Since an explicit expression of $S_f$ can hardly be determined, the idea is to use the Lemma \ref{lemma_Sf} to tackle the problem: we start by picking a certain value $y$ and by solving  $f(x)=y$. There a considerable number of methods to solve this problem. This is equivalent to solve $f(x)-y=0$, and so it is possible to apply one of the root-finding algorithms such as bisection, Newton's method, Brent's method. As a consequence, the efficiency of this method lies in the efficiency of the root-finding algorithm used. Obviously, if $y=y^*$ we are done, otherwise we know that $y^*< y$ and $x^* \in (x_{y,1},x_{y,2})$. Then we consider $\epsilon y$, where $\epsilon \in (0,1)$ and we solve $f(x)=\epsilon y$. If the problem has solution then we obtain  $x^* \in (x_{\epsilon y,1},x_{\epsilon y,2}$) which is a smaller interval than $(x_{y,1},x_{y,2})$. By proceeding in this way, it is possible to solve 
\[\min_{x \in \Omega} f(x). \]
To be more precise, the method is able to find an $x$ such that $|x-x^*|<\delta$ for any $\delta>0$.
If $f(x)=y$ has no solution, we still obtain a valuable information since we can deduce that $y^*>y$.
\begin{figure}[!h] \label{convexfuncbehav}
\centering
\subfloat{\includegraphics[scale=0.36]{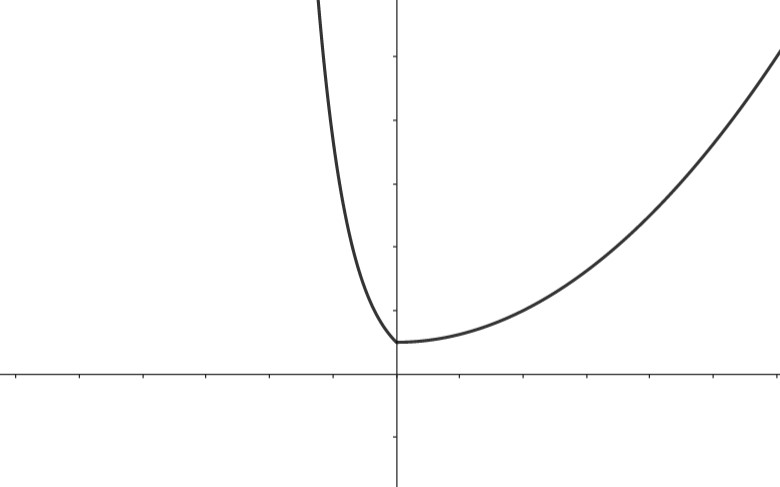}} \quad
\subfloat{\includegraphics[scale=0.36]{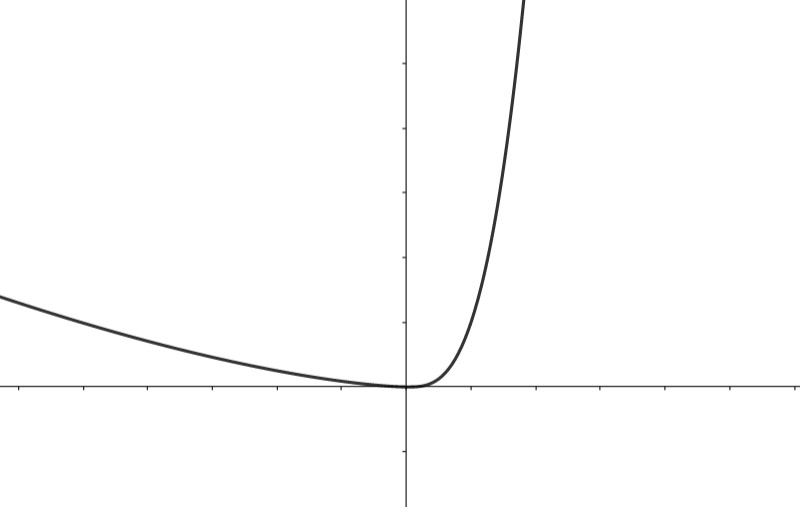}}
\caption{Esempio di due funzioni coercive}
\end{figure} \\ \noindent
The above picture presents two example of strictly convex coercive functions with different behaviour:
\[ f_1(x) = \begin{cases}
    e^{-x}  & \text{if $x<0$,}\\
    \frac{1}{16}x^2+1 & \text{if $x\geq 0$}
\end{cases} \quad
f_2(x) = \begin{cases}
    \frac{1}{8}x^{3/2}  & \text{if $x<0$,}\\
    \frac{1}{4}x^3 & \text{if $x\geq 0$}
\end{cases} 
\]
\begin{remark}
 i) It is interesting that the method presented in the upcoming Algorithm $1$ works without the assumption of differentiability of the function. Hence, it can be applied in cases where the gradient of the object function cannot be calculated. \vskip2mm \noindent
 ii) The two graphs in figure $9$ also highlight an interesting aspect of this class of function: the middle point at a certain height is influenced by the intensity of the convexity. For instance, in the first case, the function $f(x)=e^{-x}$ reaches high values faster than the function $f(x)=1/16x^2+1$. As a result, the point $(x_{y,1}+x_{y,2})/2$ will be shifted towards the right. This remark is important in computational terms. Instead of finding the two solutions of a sequence of problems of the form 
\[f(x)-y_n=0 \quad y_n \to y^*, y_n\in (y^*,+\infty),\]
one can take advantage from the different behaviour of the two tails. We can compute the two solutions $x_{y_n,1},x_{y_n,2}$ and for a fixed number of steps $h \in \mathbb{N}$, only compute $x_{y_n+i,2}$ for $i=1,\ldots,h-1$. This is because $x_{y_n+i,1}$ do not differ that much from $x_{y_n,1}$ and so they do not produce an appreciable reduction in the interval. After $h$ steps, we find the two solutions $x_{y_{n+h},1}$ and $x_{y_{n+h},2}$ and evaluate $|x_{y_{n+h},1}-x_{y_{n+h},2}|$.
\end{remark} \noindent
\begin{figure}[!h]
    \centering
    \includegraphics[scale=0.44]{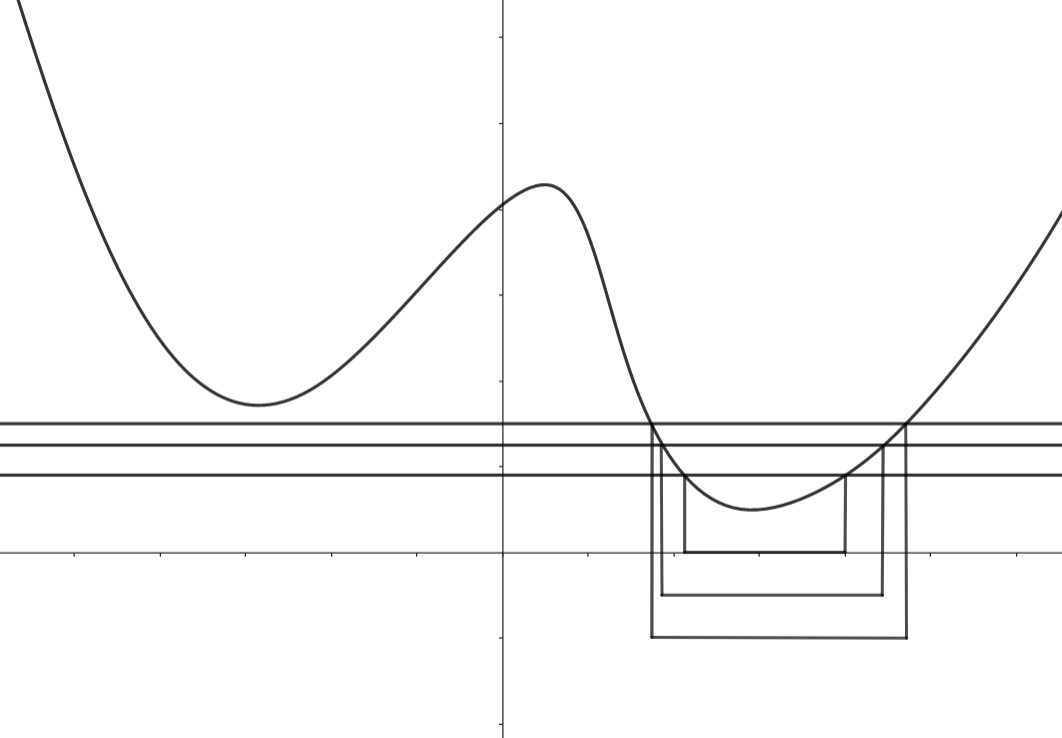}
    \caption{Visual interpretation of the method. It is underlined the reduction of the interval where the minimum lies. }
\end{figure} \\
\noindent
In light of Proposition \ref{proposizmia} and the geometric idea in Lemma \ref{lemma_Sf}, we are now ready to formalize the algorithm to find the minimum point. This algorithm works with coercive strictly convex functions or with functions that are strictly convex in a neighborhood of $x^*$. In the latter case we require to be sufficiently close to the minimum point.
We denote with $y^*$ the minimum of the function and $x^*$ the associated value.  
\begin{algorithm}[!h]
\caption{}
\begin{algorithmic}
    \STATE Given $\epsilon>0$, $\delta>0$ with $\delta \in (0,1)$, $\tilde{y}<y^*$ and 
     $y>y^*$;
     \STATE Compute $x_1,x_2$ such that $f(x_1)=f(x_2)=y$;
     \STATE Compute $|x_2-x_1|$, where $x^* \in (x_2,x_1);$  
    \WHILE {$|x_2-x_1|>\delta$} 
        \STATE $temp = (y+\tilde{y})/2$;
         Compute $x_1,x_2$ such that $f(x_1)=f(x_2)=temp$;
        \IF{$\nexists\,(x_1,x_2)$}
        \STATE $\tilde{y}=temp$;
        \ELSE
        \STATE $y=temp$;
        \ENDIF
    \ENDWHILE
 \RETURN $x_1,x_2$.

\end{algorithmic}
\end{algorithm} \vskip2mm \noindent
\textbf{Proof of correctness:}
Let us denote with $u_n$ the sequence of the upper bounds for which there exist $x_1,x_2$ and with $l_n$ the sequence of the lower bounds. At the first step, we have $u_1=y$, $l_1=\tilde{y}$ and $y^* \in (u_1,l_1)$. In the next step with end up with $u_2=u_1$ and $l_2=(u_1+l_1)/2$ or $u_2=(u_1+l_1)/2$ and $l_2=l_1$. The value $y^*$ lies in $(l_2,u_2)$ which is half of the previous one. It is immediate to conclude, recalling the bisection method, that $u_n\to y^*$ and $l_n \to y^*$ for $n\to \infty$, as a consequence we can find $x_1,x_2$ such that $|x_1-x_2|<\delta$. \\
\begin{remark}
i) The stop criterion requires a tolerance $\delta$: we can set $\delta=10^{-6}$, it depends on the algorithm we use to find the two solutions and on the precision we want to achieve. It is important to underline that if $f(x)$ is strictly convex and coercive, then $y$ can be arbitrary, otherwise we need $y$ to be sufficiently close to $y^*$. 
\end{remark} 
\noindent
\section{Generalization to a wider class of functions}
The natural question is if this idea can be generalized to a wider class of function. In other words, provided sufficient regularity of a function $f$ and given $x^*$ a strong local minimum (otherwise there is no hope to have the "v-shape" at least in a neighborhood of $x^*)$, does $f$ behave locally as coercive, strictly convex function? If the answer is positive, then we can use the previous algorithm to find the value of the minimum to a wider class of functions. Unfortunately, this is not true as shown by the following counterexample. 
\begin{example}
    Let us consider the following function:

\[ f(x) = \begin{cases}
    x^2(\sin{\frac{1}{x}}+1)   & \text{if $x\neq 0$,}\\
    0 & \text{if $x=0$.}
\end{cases}  
\]
This function is continuous and attains its global minimum at $x=0$ because $x^2\geq 0$ and $\sin{\frac{1}{x}}+1 \geq 0$. However, the function is not convex in a neighborhood of $x=0$ due to the oscillation of $\sin{\frac{1}{x}}$.

\begin{figure}[!h]
    \centering
    \includegraphics[scale=0.8]{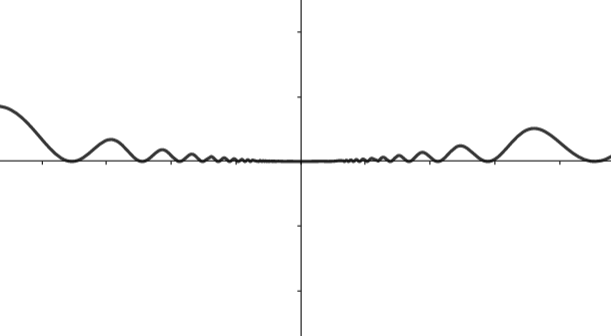}
    \caption{The oscillation of $f(x)$ prevents the function from being convex near $x=0$}
\end{figure}
\noindent
\end{example} \noindent
The problem persists even though we consider smooth function. 
\begin{example}
Let us consider $\phi(x)$ a smooth function with support $[0,1]$. We also require that $\phi(x)\geq 0$. Therefore, the function 
\[\psi(x)=\phi(x)\sin^2{\frac{1}{x}}\]
is smooth and $\psi(x)\geq 0$. Now since each term is greater or equal than zero, the minimum is $0$. At $x=0$ we have a local minimum because $\psi(0)=0$, but the function in each neighborhood of $x=0$ cannot be convex. Indeed, among the zeros of the function we have the points of the form 
\[x=\frac{1}{k\pi} \quad k\in \mathbb{Z}\] 
and 
\[\lim_{k\to +\infty} \frac{1}{k\pi}=0.\]
Therefore, in each neighborhood of zero, the function attains the value $0$ infinitely many times, which is in contrast with the definition of convexity. In other terms, if we focus our attention in the interval 
\[ \Bigl[\frac{1}{k\pi}, \frac{1}{(k+1)\pi}\Bigr]\]
for a proper $k$, here the function is positive and zero on the boundaries, therefore the segment connecting the points
$(\frac{1}{k\pi},0)$ and $(\frac{1}{(k+1)\pi},0)$ lies under the graph.
\end{example}
\noindent
The problem seems to be connected with the oscillation of the function $\sin(x)$. However, this is not true: in the following example we fail to have convexity without this term. 
\begin{example}
The function $f(x)=\sqrt{|x|}$ has a global minimum at $x=0$ but in each neighborhood of $x=0$ the function is neither convex nor concave. In this case the problem is the non existence of the derivative at $x=0$, and 
\[\lim_{x \to 0^+} \frac{d}{dx} f(x)=+\infty, \quad \lim_{x \to 0^-} \frac{d}{dx} f(x)=-\infty
\]
which prevents local convexity. 
\end{example} \noindent
The previous examples are very important because they highlight the way we can generalize the method presented in Algorithm $1$. We have to avoid terms of oscillation and the requirement of the strict convexity it is not necessary. We ask $f(x)$ to be continuous, locally coercive, to work in a proper neighborhood of the strong local minimum point $x^*$, and to have a change of monotonicity at most a finite number of times. As a consequence, from the continuity in this neighborhood, the function $f$ intersects the horizontal line $h(x)=y$ a finite number of times. Combining all the properties, we can choose the closest intersection $x_{y,i}$ for $i=1,2$ to $x^*$, and therefore proceed with the algorithm. This class of functions contains the strictly convex functions, in view of Proposition \ref{proposizionemiacorc} and Lemma \ref{Ilmiolemma}. We therefore conclude that we have extended further the class of functions to which the above algorithm can be applied.

\begin{corollary}
    Let $f(x)$ a continuous function, locally coercive and $x^*$ a strong local minimum point. If the function has a change of monotonicity at most a finite number of times, then Algorithm $1$ converges to $x^*$. 
\end{corollary}
Unfortunately, if we remove the hypothesis of the change of monotonicity, we may encounter functions with anomalous behavior that do not allow us to use the algorithm. Indeed, in order to use it, there must exist $B(x^*,\delta_y)$ and $\bar{y}$ such that
\[ \#\{(h(x)=y) \cap graph(f)\}=2 \quad \forall \,y\,\,|\,\,m=f(x^*)<y<\bar{y}.\]
\section{The construction of the anomalous test function}

Let us construct the counter-example. Let us fix $x^*\in \mathbb{R}$ and fix $f(x^*)=m>0$. We aim to construct a function $f$, with minimum at $x^*$, such that for a given $\epsilon>0$ we want to attain infinitely many times the values
\[m+\frac{\epsilon}{n} \quad n\in \mathbb{N}, n\,\,\text{odd}.\]
To this aim, fix $\delta>0$ and consider the following disjoint intervals:
\[ [x^*-\delta, x^*-\frac{\delta}{2}], [ x^*-\frac{\delta}{3}, x^*-\frac{\delta}{4}],\ldots, [x^*-\frac{\delta}{n}, x^*-\frac{\delta}{n+1}],\ldots \quad n\in \mathbb{N}, n\,\,\text{odd}.  \]
Our aim is to let the function achieve infinitely many times the value $m+\frac{\epsilon}{n}$ in $[ x^*-\frac{\delta}{n}, x^*-\frac{\delta}{n+1}]$, we also require that 
\[ f\Bigl(x^*-\frac{\delta}{n}\Bigr)=m+\frac{\epsilon}{n} \] 
for $n$ odd. 
Let us so work in $[x^*-\frac{\delta}{n}, x^*-\frac{\delta}{n+1}]$ for $n$ odd: here we want an oscillatory behavior of the function. Let us start by considering the function 
\[ h(x) = \begin{cases}
    x^2(\sin{\frac{1}{x}}+1)   & \text{if $x\neq 0$,}\\
    0 & \text{if $x=0$.}
\end{cases}  
\]
We want to adapt this function in $[x^*-\frac{\delta}{n}, x^*-\frac{\delta}{n+1}]$ according to our requests. The function must attain the value $m+\frac{\epsilon}{n}$ infinitely many times. As a consequence, we start by translating the function $h(x)$ by $m+\frac{\epsilon}{n}$. Secondly, we need to locate the zeros of $h(x)$ in the interval $[x^*-\frac{\delta}{n}, x^*-\frac{\delta}{n+1}]$. We therefore shift the function as follows:
\[ h_n(x) = \begin{cases}
    (x-(x^*-\frac{\delta}{n}))^2(\sin{\frac{1}{x-(x^*-\frac{\delta}{n})}}+1)+m+\frac{\epsilon}{n}  & \text{if $x\in (x^*-\frac{\delta}{n}, x^*-\frac{\delta}{n+1}]$,}\\
    m+\frac{\epsilon}{n} & \text{if $x=x^*-\frac{\delta}{n}$,} \\
     0 & \text{alternatively.}
\end{cases}  
\]
Furthermore, since $m=f(x^*)$ will be our point of strong local minimum, we have to control that the function is greater than $m$ at every $x\neq x^*$. However, from
\[
(x-(x^*-\frac{\delta}{n}))^2(\sin{\frac{1}{x-(x^*-\frac{\delta}{n})}}+1)\geq 0
\]
we conclude that
\[ 
h_n(x)\geq m+\frac{\epsilon}{n} \quad \forall x\in  [x^*-\frac{\delta}{n}, x^*-\frac{\delta}{n+1}]
\]
In each interval of the form $[ x^*-\frac{\delta}{n}, x^*-\frac{\delta}{n+1}]$ for $n$ even, we consider the straight line that connects 
\[  (x^*-\frac{\delta}{n}, h_{n-1}(x^*-\frac{\delta}{n})), \quad (x^*-\frac{\delta}{n+1}, m+\frac{\epsilon}{n+1}) \]
and so we have for $n$ even, the function 
\[ k_n(x)=\begin{cases}  h_{n-1}(x^*-\frac{\delta}{n})+\frac{m+\frac{\epsilon}{n+1}-h_{n-1}(x^*-\frac{\delta}{n})}{\frac{\delta}{n^2+n}}(x-(x^*-\frac{\delta}{n})) & x\in [ x^*-\frac{\delta}{n}, x^*-\frac{\delta}{n+1}], \\
 0 & \text{alternatively.}
\end{cases}  
\]
Now, we are almost done. We consider the interval $[x^*-\delta,x^*)$, here we combine the various functions for $n$ even and odd. Let us denote with $\tilde{h}(x)$ the function 
\[\tilde{h}(x)=\Bigl(\sum_{n\,\, \text{odd}} h_n(x)+ \sum_{n\,\, \text{even}} k_n(x)\Bigr) \chi [x^*-\delta,x^*).\]
By the previous construction, $\tilde{h}(x)$ is continuous in $[x^*-\delta,x^*)$, $\tilde{h}(x)>m$ and takes the values $m+\frac{\epsilon}{n}$ for $n$ odd, infinitely many times. 
We complete $\tilde{h}(x)$ with its symmetry with respect to the $x=x^*$ axis. To conclude, we consider the following function
\[ f(x)=\begin{cases} 
\tilde{h}(x) & x\in [x^*-\delta,x^*),\\
\tilde{h}(2x^*-x) & x\in (x^*, x^*+\delta] \\
m & x=x^*.
\end{cases}  
\]
that is a continuous and locally coercive function with a strong local minimum at $x^*$, for which we cannot use the algorithm.

\section*{Acknowledgements}
I would like to express my gratitude to Professors Alessia Ascanelli and Valeria Ruggiero from the University of Ferrara for their guidance and support throughout this research. Their expert advice and insights have greatly contributed to this work.

\end{document}